\documentclass[letterpaper, 10pt, conference]{ieeeconf}
\IEEEoverridecommandlockouts%
\newtheorem{theorem}{Theorem}
\newtheorem{corollary}{Corollary}
\newtheorem{lemma}{Lemma}
\newtheorem{remark}{Remark}
\newtheorem{definition}{Definition}
\newtheorem{proposition}{Proposition}
\newtheorem{example}{Example}

\usepackage{amsmath,amsfonts,amssymb,color}
\allowdisplaybreaks%
\usepackage{epsfig}
\usepackage{psfrag}
\usepackage{algorithm}
\usepackage{array}
\usepackage{epstopdf}
\usepackage{tikz, subcaption, cite}
\usepackage{booktabs}
\usepackage{amsmath}
\usepackage{dblfloatfix}

\usetikzlibrary{quotes}

\title{\LARGE \bf Optimal Strategies in a Sequential Contest with Failures}
\author{Alec~Pannunzio*, Ashley~Shaffer*, Gregory Shaver, and Shreyas Sundaram
\thanks{The authors are with the School of Electrical and Computer Engineering, and the School of Mechanical Engineering at Purdue University. The following symbol * indicates equal contribution. Emails: {\tt \{afpannun, shaffe03, gshaver, sundara2\}@purdue.edu}.}%
}
\begin{document}
\maketitle
\thispagestyle{empty}
\pagestyle{empty}
\begin{abstract}
We study a Stackelberg model where two players compete in a sequential contest, and the player that successfully invests the most effort wins. Our formulation considers a probability that each player fails to invest their chosen level of effort, as captured by a failure function that increases with the amount of effort.  The players each have their own cost of failure, as well as a maximum level of effort (capturing different abilities of the players).  We prove three key results. First, we provide the optimal strategy for Player 1 (the follower), to maximize their utility given an observed effort exerted by Player 2 (the leader). Next, we characterize the optimal strategy to maximize the expected utility of Player 2, anticipating the optimal response by Player 1. Finally, we analyze the expected utilities of each player when both play their optimal strategy, and provide 
key insights into how heterogeneous player abilities and failure costs affect the outcomes of the game.
\end{abstract}
\section{Introduction}\label{sec:introduction}
\textit{Sequential contests} are single or multi-round competitions in which players compete in a non-simultaneous, specified order~\cite{konrad2009}. The player who moves first is referred to as the Stackelberg leader, while the remaining players act as followers. Sequential contests are primarily studied from two perspectives. The contest designer's perspective focuses on how the rules and structure of the contest should be organized to achieve a desired objective. Conversely, the competitor's perspective focuses on how players should perform or invest resources in the contest. In a sequential contest, players choose strategies and receive utilities that depend on both their own strategies and those chosen by other players. 
Many real-world competitive environments can be modeled as sequential contests, including rent-seeking competitions~\cite{Linster1993} and research and development (R\&D) races~\cite{HarrisVickers1987}.

The literature on sequential contests from a contest designer's perspective focuses on how contest structures influence players’ resource investment and performance. Increasing the information available to later players about earlier players' efforts increases the total effort invested in the contest~\cite{Hinnosaar2024}. Similarly, varying the value and number of prizes~\cite{Moldovanu2001,konrad2009} affects the total effort invested in the contest. The contest designer can also add constraints on the contest itself by setting a minimum effort requirement~\cite{Sela2013} to reduce the number of participants.

The literature on sequential contests from the competitor's perspective investigates how competitors allocate resources in a competition. In a rent-seeking contest~\cite{Linster1993}, the optimal allocations of the leader and follower can be characterized based on how the follower responds to the leader's allocation. The probability of winning can also impact the amount of effort allocated, with players adjusting their effort based on their opponents' performances~\cite{Dixit1987,Baik1992}. Other research studies how a leader's pre-allocation of resources influences the subsequent players' responses~\cite{Rahul2022}. Information asymmetry can also impact resource allocation. A player's true capabilities may be unknown until their performance is observed~\cite{Grimsman2020}, or the maximum amount of resources available to each player may be assigned randomly~\cite{Paarporn2025}. Information asymmetry can arise from uncertainty in a player's performance. This is discussed in~\cite{SegevSela2014}, which models a sequential contest where the leader's effort translates into a noisy output that is observed by the follower. The follower then decides how much effort to allocate based on this observed output. While these works consider various factors that affect players' resource allocation decisions, they do not consider the possibility that a player may fail when exerting effort.

In this paper, we are interested in understanding the optimal strategies of two players in a sequential game where each player may fail while exerting their chosen level of effort with an associated cost. Failure has been incorporated into game-theoretic and resource allocation models in several different ways. In~\cite{Sundaram}, players invest in a shared resource whose probability of failure increases with the total amount invested and failure results in the loss of investment. Failure probabilities have also been modeled as dependent on the workload allocated to individual resources, with greater workloads increasing the probability of failure~\cite{Zhu2022}. Relating back to competitive settings, players may fail by being unable to participate in the competition~\cite{Lewenberg2017}. Risk has also been considered as a strategic decision in contests, where a player chooses a level of risk that determines their performance and probability of failure~\cite{Spadoni2018}. While these works consider how failure can arise from resource investment, participation, or strategic risk, they do not consider how effort-dependent failure affects the strategies of players in a sequential contest.

In our work, we use a Stackelberg model to study this setting from the competitor's perspective. We incorporate an effort-dependent failure probability and an associated failure cost into each player's utility, with the probability of failure increasing as a player exerts more effort. This captures settings such as sports competitions, where a competitor's performance is related to their exerted effort, but greater exertion can also increase the risk of injury or a crash. In such settings, the key challenge is to understand how much effort the leader should exert while accounting for the probability and cost of failure, the ability of the follower to respond, and the relative values of the prizes. The main contributions of this paper are as follows:
\begin{enumerate}
    \item We provide a complete characterization of the optimal strategies for the leader and follower that maximize their individual utilities while accounting for the possibility of failure and associated costs.  
    \item We characterize the expected utilities of each player when they are playing their optimal strategies, and provide insights into how heterogeneous player abilities and failure costs impact the outcomes of the game.
\end{enumerate}
This paper is organized as follows. Section~\ref{sec:Problem_Formulation} details the problem formulation and formally defines our failure function. Section~\ref{sec:Optimal_Strategies} characterizes the utilities of each player and analyzes each player's optimal strategies based on the game model. Section~\ref{sec:utility_discussion} defines the expected utilities of each player when they are playing their optimal strategy along with numerical results for discussion. The paper is summarized and concluded in Section~\ref{sec:Conclusion}.

\section{Problem Formulation}\label{sec:Problem_Formulation}
We consider a sequential contest using a Stackelberg model in which there are two rational players, denoted Player 1 and Player 2, who must each choose an amount of effort $x_1 \in \mathbb{R}_{\ge 0}$ and $x_2 \in \mathbb{R}_{\ge 0}$, respectively, to invest in the game.  Each player $i \in \{1, 2\}$ has an associated \textit{failure function} $f_i(x)$ which maps each effort value, $x\in \mathbb{R}_{\ge 0}$, to the probability that the player fails to achieve that effort $x$. A more detailed discussion is provided in Section~\ref{subsec:failurefunction}. Each player $i$ has a failure cost defined as $c_i\in \mathbb{R}_{\ge 0}$ to capture injury or damage to assets in the event of failure.
Player 2 moves first in the game (takes on the role of leader) and Player 1 moves second (takes on the role of follower) after observing the outcome of Player 2's choice. If Player 2 fails to invest their chosen effort, Player 1 automatically wins. Conversely, if Player 2 successfully invests their effort and Player 1 fails to invest more or equal effort, Player 2 wins. Player 1 or Player 2 may choose to ``fold'' by investing 0 effort into the game and avoiding the risk of the failure cost. The player who successfully invests more effort will win and any ties are broken in favor of Player 1. We award $W_2\in\mathbb{R}_{\ge 0}$ for losing and $W_1\in\mathbb{R}_{> W_2}$ for winning. This notation resembles the final round in an elimination tournament~\cite{Rosen1986}.

\subsection{Failure Function}\label{subsec:failurefunction}
We suppose that each player $i \in \{1,2\}$ has an individual maximum level of effort that they can exert, denoted by $x_{i_{\text{max}}} \in \mathbb{R}_{> 0}$.  We will investigate failure functions of the 
form:
\begin{align}
f_i(x) =
\begin{cases}
\left(\frac{x}{x_{i_{\text{max}}}}\right)^r, & 0 \leq x < x_{i_{\text{max}}},\\
1, & x\geq x_{i_{\text{max}}},
\label{eqn:failure function}
\end{cases}
\end{align}
where the exponential parameter $r\in \mathbb{R}_{\ge 1}$ can be used to tune how quickly the probability of failure approaches 1 as $x$ approaches $x_{i_{\text{max}}}$.

This function has a few key properties that make it a useful choice for analysis. It allows us to specify a range of efforts ${[0, x_{i_{\text{max}}}]}$ that a given player may exert. The strictly increasing failure probability is zero with zero effort and 1 with $x_{i_{\text{max}}}$ effort. Additionally, with a high $r$ parameter, the function allows us to represent a player that has a relatively low chance of failure until they approach their $x_{i_{\text{max}}}$. Conversely, we can model a player with a linear increase in failure probability as they approach their $x_{i_{\text{max}}}$ with an $r$ value of 1. As $r$ approaches $\infty$, the failure function approaches:
\begin{equation*}
    f_i(x) = \begin{cases}
        0, & x<x_{i_{\text{max}}}, \\
        1, & x \geq x_{i_{\text{max}}}.
    \end{cases}
\end{equation*}
This removes all uncertainty and makes the game deterministic. In this case, each player can confidently exert effort up to their $x_{i_{\text{max}}}$. Note that by allowing each player to have their own maximum effort value $x_{i_{\text{max}}}$, our game will allow us to investigate the impact of heterogeneous player abilities on the outcome of the game. Note we also assume that both players have the same parameter $r$ in their failure function for ease of analysis; as we will see, the game will already present significant complexities under this assumption. Nevertheless, we will be able to obtain substantial insights that we anticipate will carry forward to more general settings. Finally, we assume failure for each player is independent of each other.

\section{Optimal Strategies for each Player}\label{sec:Optimal_Strategies}

We will now describe the utility functions for each player in this game, and provide the corresponding optimal strategies for each player.  

\subsection{Player 1 Analysis}\label{subsec:player_1_analysis}
Player 1 moves second and will make their move with full knowledge of the effort invested and resulting outcome (success or failure) of Player 2. Thus, we begin by looking at the best response of Player 1 to any strategy chosen by Player 2. If Player 2 successfully invests some effort $x_2$, Player 1's expected utility 
is given by the following piecewise function:
\begin{multline}
E[u_1(x)]\\=
\begin{cases}
W_2-c_1f_1(x),&0\le x<x_2 \\
W_1(1-f_1(x))+(W_2-c_1)f_1(x),&x\ge x_2
\end{cases}.
\label{eqn:player_1_utility}
\end{multline}
Specifically, if Player 2 has successfully invested $x_2$, then for any investment $0 \le x < x_2$ by Player 1, they will receive a prize of $W_2$ (for losing), and pay a failure cost of $c_1$ in the event that they fail (which happens with probability $f_1(x)$). This yields the expected utility in the first case shown above.  On the other hand, if Player 1 chooses to invest an effort $x \ge x_2$, then they will win with probability $1-f_1(x)$ and receive the prize $W_1$, or fail with probability $f_1(x)$ and pay the failure cost of $c_1$ while winning the second place prize of $W_2$. This yields the expected utility shown in the second case above.

The following theorem provides the optimal strategy for Player 1 to maximize their expected utility~\eqref{eqn:player_1_utility}.

\begin{theorem}\label{thrm:player_1_strategy}
Define $\bar{c}_1 \triangleq \frac{c_1}{W_1-W_2}$.  If Player 2 successfully invests an effort of $x_2$, Player 1 should invest the following to maximize their utility:
\begin{align*}
x_1^*=
\begin{cases}
x_2, & x_2 < x_{\text{threshold}}, \\
0, & \text{otherwise}, \\
\end{cases}
\end{align*}
\text{where }
\begin{equation}
x_{\text{threshold}} \triangleq {x_{1_{\text{max}}}}\sqrt[r]{\frac{1}{1+\bar{c}_1}} \enspace.
\label{eqn:threshold_effort}
\end{equation}
\end{theorem}
\begin{proof}
Suppose Player 2 has successfully invested a nonzero effort of $x_2$. Examining Player 1's expected utility~\eqref{eqn:player_1_utility}, we see that it is decreasing in $x$ for each of the two ranges; thus, the value of $x$ that maximizes the utility in each of the two ranges is given by $x = 0$ (for $0 \le x < x_2$) and $x = x_2$ (for $x \ge x_2$).  Thus, the first case captures folding, and the second case indicates challenging by matching Player 2's successfully invested effort.

We will compare the expected utility of the optimal effort of each case from equation~\eqref{eqn:player_1_utility} to determine when challenging yields the higher expected utility.  Specifically, from~\eqref{eqn:player_1_utility}, when $x = 0$, the expected utility is given by $W_2$ (the losing prize), and when $x = x_2$, the expected utility is $W_1(1-f_1(x_2))+(W_2-c_1)f_1(x_2)$.  Thus, challenging (by playing $x = x_2$) yields the higher expected utility when
\begin{align*}
&W_2 < W_1(1-f_1(x_2))+(W_2-c_1)f_1(x_2) \nonumber\\
\Leftrightarrow\quad &f_1(x_2)(W_1-W_2+c_1)<W_1-W_2 \nonumber \\
\Leftrightarrow\quad &f_1(x_2)<\frac{1}{1+\bar{c}_1},
\end{align*}
where $\bar{c}_1=\frac{c_1}{W_1-W_2}$.
This is the condition at which Player 1 will challenge Player 2's investment. Using~\eqref{eqn:failure function}, we can solve for the threshold level of Player 2's effort that will determine whether Player 1 will challenge or fold as follows:
\begin{align*}
x_{\text{threshold}}&=f_1^{-1}\left(\frac{1}{1+\bar{c}_1}\right)=x_{1_{\text{max}}}\sqrt[r]{\frac{1}{1+\bar{c}_1}}.
\end{align*}
\end{proof}

{\bf Key insights:} Theorem~\ref{thrm:player_1_strategy} shows if Player 2 is able to invest the threshold effort $x_{\text{threshold}}$ without failing, they will win the game because Player 1 will fold. Player 1 will challenge if the effort Player 2 invests is lower than this threshold value. In this scenario, Player 2 will only win if Player 1 fails to invest that effort. Recall that $\bar{c}_1$ is the ratio of Player 1's failure cost to the relative gain for winning. As this ratio increases, Player 2 has to successfully invest a smaller level of effort in order to force Player 1 to fold. Similarly, note that as $\bar{c}_1\to0$, $x_{\text{threshold}}$ approaches the point $x_{1_{\text{max}}}$ where failure is certain, $f_1(x_{\text{threshold}})=1$ (i.e., when Player 1 does not care about failing, they will push to their maximum possible effort to win). Finally, note that as $r$ increases (i.e., the failure function is relatively flat until the effort is very close to $x_{i_{\text{max}}}$), Player 1's threshold for challenging approaches $x_{1_{\text{max}}}$, as expected.

\subsection{Player 2 Analysis}\label{subsec:player_2_analysis}
Having characterized Player 1's optimal response to Player 2's exerted effort, we will now provide Player 2's utility (anticipating Player 1's response).  

Recall from Theorem~\ref{thrm:player_1_strategy} that Player 1's best strategy will be to either match Player 2's effort or fold and invest 0 effort. This choice will be based solely on whether Player 2's invested effort is greater than the threshold effort $x_{\text{threshold}}$ given in~\eqref{eqn:threshold_effort}. Utilizing this result, the expected utility for Player 2 attempting to invest some effort $x_2$ is given by the following piecewise function:
\begin{multline*}
E[u_2(x_2)] \\= 
\begin{cases}
W_1 f_1(x_2)(1 - f_2(x_2)) \\
\quad + W_2 (1 - f_1(x_2))(1 - f_2(x_2)) \\
\quad + (W_2 - c_2) f_2(x_2), 
&x_2 < x_{\text{threshold}}, \\
W_1 (1 - f_2(x_2)) \\
\quad+ (W_2 - c_2) f_2(x_2), 
&x_2 \ge x_{\text{threshold}}.
\end{cases}
\end{multline*}
Specifically, if Player 2 successfully invests $x_2 < x_{\text{threshold}}$, then Player 2 will receive $W_1$ if they succeed (which happens with probability $1-f_2(x_2)$) and Player 1 fails (which happens with probability $f_1(x_2)$). Player 2 will receive $W_2$ with this investment if neither of the players fail or Player 2 fails, in which case they also pay a failure cost of $c_2$. This yields the expected utility in the first case shown above. Conversely, if Player 2 attempts to invest $x_2\ge x_{\text{threshold}}$, they will succeed and receive $W_1$ with probability $1-f_2(x_2)$ or fail and receive $W_2$ with probability $f_2(x_2)$ and pay a failure cost of $c_2$. Recall from Theorem~\ref{thrm:player_1_strategy} that Player 1 will not challenge because $x_2 \ge x_\text{threshold}$.

For convenience, we define the quantity \begin{align*}
    \bar{u}_2(x_2) \triangleq \frac{E[u_2(x_2)]-W_2}{W_1-W_2},
\end{align*}
and note that $\bar{u}_2(x_2)$ has the same maximizer as $E[u_2(x_2)]$ (since subtracting a constant and scaling does not change the maximizer of a function).  Further defining $\bar{c}_2 \triangleq \frac{c_2}{W_1-W_2}$, we have
\begin{multline}
\bar{u}_2(x_2)\\ =
\begin{cases}
f_1(x_2)(1-f_2(x_2))\\
\qquad-\bar{c}_2f_2(x_2), &0\le x_2<x_{\text{threshold}}, \\ 
1-f_2(x_2)-\bar{c}_2f_2(x_2), &x_2\ge x_{\text{threshold}}.
\end{cases}
\label{eqn:player_2_utility}
\end{multline}

Player 2's optimal strategy is therefore obtained by maximizing this function.  The following theorem (our main result in this paper) completely characterizes this optimal strategy for Player 2.

\begin{theorem}\label{thrm:player_2_strategy}
Define the quantities $\bar{c}_1 \triangleq \frac{c_1}{W_1-W_2}$, $\bar{c}_2 \triangleq \frac{c_2}{W_1-W_2}$ and $\alpha \triangleq \left(\frac{x_{1_{\text{max}}}}{x_{2_{\text{max}}}}\right)^r$.  
Player 2's optimal strategy ($x_2^*$) is given as follows.

\noindent If  $\bar{c}_1\bar{c}_2 \ge 1$:
\begin{equation*}\label{equationc1c2case1}
x_2^* =
\begin{cases}
x_{1_{\text{max}}}\sqrt[r]{\frac{1}{1+\bar{c}_1}}, & 0 \le \alpha < \frac{1+\bar{c}_1}{1+\bar{c}_2}, \\
0, & \frac{1+\bar{c}_1}{1+\bar{c}_2} \le \alpha.
\end{cases}
\end{equation*}
\noindent Otherwise, if $\bar{c}_1\bar{c}_2 < 1$:
\begin{equation*}
x_2^* =
\begin{cases}
x_{1_{\text{max}}}\sqrt[r]{\frac{1}{1+\bar{c}_1}}, & 0 \le \alpha < \alpha^*, \\
x_{2_{\text{max}}}\sqrt[r]{\frac{1}{2}(1 - \bar{c}_2 \alpha)}, & \alpha^* \le \alpha < \frac{1}{\bar{c}_2}, \\
0, & \frac{1}{\bar{c}_2} \le \alpha,
\end{cases}
\end{equation*}
\noindent \text{where}
\begin{equation*}
\alpha^* = \frac{
2 + \bar{c}_2 + 2\sqrt{\frac{\bar{c}_1(1+\bar{c}_2)}{1+\bar{c}_1}}
}{
4\frac{1+\bar{c}_2}{1+\bar{c}_1} + \bar{c}_2^2
}.
\end{equation*}
\end{theorem}

The proof of the above theorem requires careful analysis of several conditions, and thus we will build towards the proof in the following subsection.  

{\bf Key insights: } Theorem~\ref{thrm:player_2_strategy} provides key insights into Player 2's optimal strategy in this game. The quantity $\alpha$ captures the relative abilities of the two players with $\alpha < 1$ if Player $1$ is weaker than Player 2 (i.e., $x_{1_{\text{max}}} < x_{2_{\text{max}}}$) and vice versa. We assume that Player 2 has full knowledge of $x_{1_{\text{max}}}$. There are three distinct possibilities for Player 2's optimal strategy that correspond with a range of values for $\alpha$. In the lower range of $\alpha$ (namely $0\le\alpha<\frac{1+\bar{c}_1}{1+\bar{c}_2}$ for $\bar{c}_1\bar{c}_2\ge1$ and $0\le\alpha<\alpha^*$ for $\bar{c}_1\bar{c}_2<1$), Player 2 will choose the value for $x_\text{threshold}$ from Theorem~\ref{thrm:player_1_strategy}. This corresponds to the subset of games where Player 1 is sufficiently weaker than Player 2, so Player 2 can confidently invest a sufficiently high effort such that Player 1 will not benefit from challenging.
If $\alpha$ is sufficiently large however (i.e., $\alpha\ge\frac{1+\bar{c}_1}{1+\bar{c}_2}$ for $\bar{c}_1\bar{c}_2\ge1$ and $\alpha\ge\frac{1}{\bar{c}_2}$ for $\bar{c}_1\bar{c}_2<1$), then Player 2 will fold. In this subset of games, Player 1 is so much better than Player 2 that no investment by Player 2 will yield a chance of winning $W_1$ that offsets the cost of failure. From an alternative perspective, as $\bar{c}_2\to\infty$ (with finite $\bar{c}_1$), both $\frac{1+\bar{c}_1}{1+\bar{c}_2}$ and $\frac{1}{\bar{c}_2}$ go to zero. This means Player 2's cost of failure is so high that even in a game where they are significantly stronger than Player 1, they would not chance paying such a high failure cost.

Only for $\bar{c}_1\bar{c}_2<1$ will there be a middle range of $\alpha$ where neither player will fold. At the boundary $\alpha=\alpha^*$, Player 2 receives the same utility from choosing $\hat{x}$ in \eqref{eqn:x_hat} or $x_\text{threshold}$ and inducing Player 1 to fold. We assume that Player 2 chooses $\hat{x}$ when indifferent. We note that $\bar{c}_1=1, \bar{c}_2=1$ is a case that $\bar{c}_1\bar{c}_2=1$ where each player's failure cost is equal to the utility gained by winning $W_1$ over $W_2$. We see that if these quantities are sufficiently high (as is in the case of $\bar{c}_1\bar{c}_2\ge1$), for any value of $\alpha$, at least one of the players will not be willing to chance failure, thus folding. We conclude that there is a seemingly narrow range of game parameters where neither player will fold, which is the case when the full range of failure costs is considered. In competitive settings where the additional reward from winning exceeds the cost associated with failure, $\bar{c}_1$ and $\bar{c}_2$ will usually be less than 1, making $\bar{c}_1\bar{c}_2<1$. Additionally, this intermediate range of $\alpha$ can arise when players are evenly matched. However, it is important to note that large values of $r$ will cause drastic changes in $\alpha$ for small differences in $x_{1_\text{max}}$ and $x_{2_\text{max}}$. As $r\to\infty$ (i.e., players are able to invest close to their maximum effort with little chance of failure), $\alpha$ will go to $\infty$ if $x_{1_\text{max}}>x_{2_\text{max}}$ and 0 if $x_{1_\text{max}}<x_{2_\text{max}}$.

\subsection*{Building the Proof of Theorem~\ref{thrm:player_2_strategy}}
We will analyze Player 2's utility~\eqref{eqn:player_2_utility} to identify the optimal strategy for each of the two ranges.  Specifically, we will establish a set of lemmas that characterize the optimal strategy (and corresponding utility) in each case.  We start with the following lemma focusing on the simpler case of $x_2 \ge x_{\text{threshold}}$.
\begin{lemma}\label{lma:x_thresh}
In the range $x_2 \ge x_{\text{threshold}}$, the quantity $\bar{u}_2(x_2)$ is maximized at $x_2=x_{\text{threshold}}$ and offers a maximal expected utility of:
\begin{equation}
    \bar{u}_2(x_{\text{threshold}})=1-\alpha\left(\frac{1+\bar{c}_2}{1+\bar{c}_1}\right).
\label{eqn:u_two_x_thresh_w_alpha}
\end{equation}
This expected utility is positive if and only if
\begin{equation}
    \alpha < \frac{1+\bar{c}_1}{1+\bar{c}_2}.
\label{eqn:xthresh_lt_xopt_cond}
\end{equation}
\end{lemma}

\begin{proof}
From the second case (corresponding to $x_2 \ge x_{\text{threshold}}$) in~\eqref{eqn:player_2_utility}, we see that $\bar{u}_2(x_2)$ is decreasing in $x_2$ (since the failure function $f_2(x_2)$ is increasing in $x_2$). Therefore, the optimal value is the smallest $x_2$ that satisfies this case, namely $x_2 = x_{\text{threshold}}$. The value for $x_{\text{threshold}}$ from~\eqref{eqn:threshold_effort} can be substituted into the failure function~\eqref{eqn:failure function} and Player 2's utility~\eqref{eqn:player_2_utility} to get 
\begin{align*}
f_2(x_{\text{threshold}})&= \left(\frac{x_{\text{threshold}}}{x_{2_{\text{max}}}}\right)^r=\alpha\left(\frac{1}{1+\bar{c}_1}\right), \nonumber \\
\bar{u}_2(x_{\text{threshold}})&=1-(1+\bar{c}_2)\alpha\left(\frac{1}{1+\bar{c}_1}\right)\nonumber\\
&=1-\alpha\left(\frac{1+\bar{c}_2}{1+\bar{c}_1}\right).
\end{align*}
This is clearly positive if and only if condition~\eqref{eqn:xthresh_lt_xopt_cond} is satisfied. 
\end{proof}

The next lemma considers the optimal value of Player 2's utility~\eqref{eqn:player_2_utility} in the first case (namely $0 \le x_2 < x_{\text{threshold}}$).  

\begin{lemma}\label{lma:x_hat}
If $\alpha \le \frac{1}{\frac{2}{1+\bar{c}_1}+\bar{c}_2}$, the function
$\bar{u}_2(x_2)$ is increasing in $x_2$ in the range $0\le x_2 <x_{\text{threshold}}$, and the optimal value of $\bar{u}_2(x_2)$ is $x_\text{threshold}$ as defined in Theorem~\ref{thrm:player_1_strategy}.  Otherwise, the maximizer $\hat{x}_2^*$ of 
$\bar{u}_2(x_2)$ in the range $0\le x_2 <x_{\text{threshold}}$ occurs at:
\begin{align*}
    \hat{x}_2^* =
    \begin{cases}
        \hat{x}, & \alpha < \frac{1}{\bar{c}_2},\\
        0, & \text{otherwise},
        \label{eqn:optimal_x}
    \end{cases}
\end{align*}
where \begin{equation}
    \hat{x} \triangleq x_{2_{\text{max}}}\sqrt[r]{\frac{1}{2}(1-\bar{c}_2\alpha)} \enspace .
\label{eqn:x_hat}
\end{equation} 
The corresponding utility is given by
\begin{equation*}
    \bar{u}_2(\hat{x}_2^*)=
    \begin{cases}
        \frac{1}{4\alpha}(1-\alpha\bar{c}_2)^2, & \alpha < \frac{1}{\bar{c}_2}, \\
        0, & \text{otherwise}.
    \end{cases}
\end{equation*}
\end{lemma}
\begin{proof}
From~\eqref{eqn:player_2_utility}, the range of $x_2$ considered in this lemma means Player 2's utility takes the following form:
\begin{equation*}
    \bar{u}_2(x_2) = f_1(x_2)(1-f_2(x_2))-\bar{c}_2f_2(x_2).
\end{equation*}
We first compute the derivative of this utility and set it equal to zero to find the critical points. 
\begin{align*} 
0&=f_1'(x_2)-f_2(x_2)f_1'(x_2)-f_1(x_2)f_2'(x_2)-\bar{c}_2f_2'(x_2)\nonumber\\
&=(1-f_2(x_2))f_1'(x_2)-(f_1(x_2)+\bar{c}_2)f_2'(x_2)\\
&=\left(1-\frac{x_2^r}{x_{2_{\text{max}}}^r}\right)\left(\frac{r}{x_{1_{\text{max}}}}\right)\left(\frac{x_2^{r-1}}{x_{1_{\text{max}}}^{r-1}}\right)\nonumber\\
&\qquad-\left[\left(\frac{x_2^r}{x_{1_{\text{max}}}^r}+\bar{c}_2\right)\left(\frac{r}{x_{2_{\text{max}}}}\right)\left(\frac{x_2^{r-1}}{x_{2_{\text{max}}}^{r-1}}\right)\right]\nonumber\\
0 &=\left(1-\frac{x_2^r}{x_{2_{\text{max}}}^r}\right)\frac{1}{x_{1_{\text{max}}}^r}-\left[\left(\frac{x_2^r}{x_{1_{\text{max}}}^r}+\bar{c}_2\right)\frac{1}{x_{2_{\text{max}}}^r}\right].\nonumber\\
\end{align*}
In the last expression above, we divided out $x^{r-1}$, making 0 a possible optimizer if $r\neq1$.  Now, multiplying both sides by $x_{1_{\text{max}}}^r$, we obtain
\begin{align*}
0&=\left(1-\frac{x_2^r}{x_{2_{\text{max}}}^r}\right)-\alpha\left(\frac{x_2^r}{x_{1_{\text{max}}}^r}+\bar{c}_2\right)\nonumber\\
&=1-\frac{x_2^r}{x_{2_{\text{max}}}^r}-\frac{x_2^r}{x_{2_{\text{max}}}^r}-\alpha\bar{c}_2,\nonumber\\
\Leftrightarrow x_2^r&=\frac{x_{2_{\text{max}}}^r\left(1-\alpha\bar{c}_2\right)}{2}.\nonumber
\end{align*}
Denote 
\[
\hat{x} \triangleq x_{2_{\text{max}}}\sqrt[r]{\frac{1}{2}\left(1-\alpha\bar{c}_2\right)},
\]
as the solution to this equation. Thus, the optimizer in this region, $\hat{x}_2^*$, will be either \eqref{eqn:x_hat} or 0. If $\hat{x}_2^*$ is in the range $0\le\hat{x}<x_\text{threshold}$, it will be the optimizer of $\bar{u}_2(x)$ in that region. The value $\hat{x}$ will only be a valid solution if $\alpha\le\frac{1}{\bar{c}_2}$, otherwise, 0 will be the optimizer over $0\le x<x_\text{threshold}$. We note that at $\alpha=\frac{1}{\bar{c}_2}$, $\hat{x}=0$, so there is continuity in $\hat{x}_2^*$ over $\alpha$. We now compare the value of $\hat{x}$ and $x_{\text{threshold}}$ to obtain the condition on $\alpha$ where $\hat{x}_2^*$ is in the range $0\le\hat{x}<x_\text{threshold}$:
\begin{align}\label{eqn:xhat_smaller_than_xth}
&x_{2_{\text{max}}}\sqrt[r]{\frac{1}{2}(1-\bar{c}_2\alpha)}<{x_{1_{\text{max}}}}\sqrt[r]{\frac{1}{(1+\bar{c}_1)}} \nonumber\\
&\Leftrightarrow x_{2_{\text{max}}}^r\frac{1}{2}(1-\bar{c}_2\alpha)<{x_{1_{\text{max}}}^r}\frac{1}{(1+\bar{c}_1)}\nonumber\\
&\Leftrightarrow \frac{1}{2\alpha}(1-\bar{c}_2\alpha)<\frac{1}{(1+\bar{c}_1)}\nonumber\\
&\Leftrightarrow \frac{1}{\alpha}-\bar{c}_2<\frac{2}{(1+\bar{c}_1)}\nonumber\\
&\Leftrightarrow \alpha>\frac{1}{\frac{2}{1+\bar{c}_1}+\bar{c}_2}. 
\end{align}
This value is clearly smaller than $\frac{1}{\bar{c}_2}$, so if $\alpha\le\frac{1}{\frac{2}{1+\bar{c}_1}+\bar{c}_2}$, we have $\hat{x}_2^*=\hat{x}\ge x_\text{threshold}$. For that case $\bar{u}_2(x_2)$ is increasing with $x_2$ in the range $0<x<x_\text{threshold}$, and $\hat{x}_2^*$ is $x_\text{threshold}$ as defined in Theorem~\ref{thrm:player_1_strategy}. 

The utility $\bar{u}_2(\hat{x}_2^*)$ is zero when $\hat{x}_2^* = 0$, so we would like to find the resulting value of $\bar{u}_2(\hat{x}_2^*)$ for the range of $\alpha$ where $\hat{x}_2^*=\hat{x}$. We evaluate $\bar{u}_2(\hat{x})$ to find:
\begin{align*}
f_1(\hat{x})&= \left(\frac{\hat{x}}{x_{1_{\text{max}}}}\right)^r=\frac{1}{2}\left(\frac{1}{\alpha}-\bar{c}_2\right), \nonumber \\
f_2(\hat{x})&= \left(\frac{\hat{x}}{x_{2_{\text{max}}}}\right)^r=\frac{1}{2}(1-\alpha\bar{c}_2), \nonumber \\
\bar{u}_2(\hat{x})&=\frac{1}{2}\left(\frac{1}{\alpha}-\bar{c}_2\right)\left(1-\frac{1}{2}(1-\alpha\bar{c}_2)\right)-\frac{\bar{c}_2}{2}(1-\alpha\bar{c}_2)\nonumber\\
&=\frac{1}{4}\left(\frac{1}{\alpha}+\bar{c}_2-\bar{c}_2-\alpha\bar{c}_2^2\right)-\frac{\bar{c}_2}{2}+\frac{\alpha\bar{c}_2^2}{2}\nonumber\\
&=\frac{1}{4\alpha}(1-2\alpha\bar{c}_2+\alpha^2\bar{c}_2^2).\nonumber
\end{align*}
We thus 
get the following equation for $\bar{u}_2(\hat{x})$ when $\alpha<\frac{1}{\bar{c}_2}$:
\begin{equation} \label{eqn:u2_xhat_alpha_less_1_c2bar}
    \bar{u}_2(\hat{x})=\frac{1}{4\alpha}(1-\alpha\bar{c}_2)^2.
\end{equation}

We see that this utility is always nonnegative. This concludes our proof of Lemma~\ref{lma:x_hat}.
\end{proof}

\begin{lemma}\label{lma:root_calculation}
In the range $\frac{1}{\frac{2}{1+\bar{c}_1}+\bar{c}_2}<\alpha\le\frac{1}{\bar{c}_2}$, if $\frac{1+\bar{c}_1}{1+\bar{c}_2}\le\frac{1}{\bar{c}_2}$, the quantities $\bar{u}_2(x_{\text{threshold}})$ and $\bar{u}_2(\hat{x})$ are equal only when $\alpha$ is given by the following value:
\begin{align*}
\alpha^*\triangleq\frac{2+\bar{c}_2+2\sqrt{\frac{\bar{c}_1(1+\bar{c}_2)}{1+\bar{c}_1}}}{4\frac{1+\bar{c}_2}{1+\bar{c}_1}+\bar{c}_2^2}.
\end{align*}
Otherwise, they are not equal for any $\alpha$ in that range.  Furthermore, 
\begin{equation*}
    \alpha^*=\frac{1}{\bar{c}_2}
\end{equation*}
if and only if $\frac{1+\bar{c}_1}{1+\bar{c}_2}=\frac{1}{\bar{c}_2}$.
\end{lemma}
\begin{proof}
We check to see when $\bar{u}_2(x_{\text{threshold}})$ is equal to $\bar{u}_2(\hat{x})$ by equating~\eqref{eqn:u_two_x_thresh_w_alpha} and~\eqref{eqn:u2_xhat_alpha_less_1_c2bar} and multiplying both sides by $\alpha$:
\begin{align}
    -\alpha^2\left(\frac{1+\bar{c}_2}{1+\bar{c}_1}\right) + \alpha=\frac{1}{4}(1-\alpha\bar{c}_2)^2.
\label{eqn:utility_xhat_smaller_than_xth}
\end{align}
Note from Lemma~\ref{lma:x_hat} that this is only valid for $\alpha\le\frac{1}{\bar{c}_2}$. We are able to consider the case $\alpha=\frac{1}{\bar{c}_2}$ as part of this analysis because the value of equation~\eqref{eqn:u2_xhat_alpha_less_1_c2bar} is equal to $\bar{u}_2(\hat{x}^*)$ from Lemma~\ref{lma:x_hat} at this value of $\alpha$. The left-hand side is a negative quadratic and the right-hand side is a positive quadratic in $\alpha$. These are concave and convex respectively. We find the points where these two parabolas intersect through a straightforward application of the quadratic formula, which yields 
\begin{align*}
\alpha=\frac{2+\bar{c}_2\pm2\sqrt{\frac{\bar{c}_1(1+\bar{c}_2)}{1+\bar{c}_1}}}{4\frac{1+\bar{c}_2}{1+\bar{c}_1}+\bar{c}_2^2}.
\end{align*}
We denote the larger of the two possible roots by
\begin{align*}
\alpha^*
\triangleq
\frac{2+\bar{c}_2+2\sqrt{\frac{\bar{c}_1(1+\bar{c}_2)}{1+\bar{c}_1}}}{4\frac{1+\bar{c}_2}{1+\bar{c}_1}+\bar{c}_2^2}.
\end{align*}
At $\alpha=0$, the left-hand side of~\eqref{eqn:utility_xhat_smaller_than_xth} is clearly less than the right-hand side. In Appendix~\ref{subsec:plug_xth_xopt_cond_into_utility}, we show that the opposite is true when $\alpha=\frac{1}{\frac{2}{1+\bar{c}_1}+\bar{c}_2}$, and thus one of the two equality points must be between these two values for $\alpha$. At $\alpha=\frac{1}{\bar{c}_2}$, the right side of~\eqref{eqn:utility_xhat_smaller_than_xth} is zero, and the left-hand side will be larger than the right side if $\frac{1}{\bar{c}_2}<\frac{1+\bar{c}_1}{1+\bar{c}_2}$, smaller if $\frac{1}{\bar{c}_2}>\frac{1+\bar{c}_1}{1+\bar{c}_2}$, and equal if $\frac{1}{\bar{c}_2}=\frac{1+\bar{c}_1}{1+\bar{c}_2}$. In the first case the left-hand side is larger, so the second intersection point must be greater than $\frac{1}{\bar{c}_2}$. In the second case, the right-hand side is larger, so the second intersection point must be in $\frac{1}{\frac{2}{1+\bar{c}_1}+\bar{c}_2}<\alpha<\frac{1}{\bar{c}_2}$. In the last case $\alpha^*=\frac{1}{\bar{c}_2}$.
\end{proof}
We are now in place to provide the proof of Theorem~\ref{thrm:player_2_strategy}.
\subsection{Theorem 2 Proof}\label{subsec:Theorem 2 Proof}
\begin{proof}
We know that there are two possible values, one for each range of $x_2$ in equation~\eqref{eqn:player_2_utility}, that may be the optimal strategy for Player 2. We will formulate Player 2's optimal choice between these two values based on $\alpha$, $\bar{c}_1$, and $\bar{c}_2$. We will start at $\alpha>\text{max}\{\frac{1}{\bar{c}_2},\frac{1+\bar{c}_1}{1+\bar{c}_2}\}$ and use the conditions from Lemmas~\ref{lma:x_thresh},~\ref{lma:x_hat}, and~\ref{lma:root_calculation} to find points in $\alpha$ where the optimal strategy changes. 

At $\alpha>\text{max}\{\frac{1}{\bar{c}_2},\frac{1+\bar{c}_1}{1+\bar{c}_2}\}$, $\alpha > \frac{1}{\bar{c}_2}$ and $\alpha>\frac{1+\bar{c}_1}{1+\bar{c}_2}$. This means that $\bar{u}_2(x_{\text{threshold}})$ is negative by Lemma~\ref{lma:x_thresh} and $\bar{u}_2(\hat{x})=0$ by Lemma~\ref{lma:x_hat}, so $x_2^*=\hat{x}^*=0$ is the optimal strategy. 

To analyze values of $\alpha\le \text{max}\{\frac{1}{\bar{c}_2},\frac{1+\bar{c}_1}{1+\bar{c}_2}\}$, we need to know whether $\alpha=\frac{1}{\bar{c}_2}$ or $\alpha=\frac{1+\bar{c}_1}{1+\bar{c}_2}$ is larger. If we compare the two, we get the following:
\begin{align*}
\frac{1}{\bar{c}_2}&<\frac{1+\bar{c}_1}{1+\bar{c}_2}\nonumber\\
\Leftrightarrow 1+\bar{c}_2&<\bar{c}_2+\bar{c}_1\bar{c}_2\nonumber\\
\Leftrightarrow 1&<\bar{c}_1\bar{c}_2,
\end{align*}
which additionally shows that $\bar{c}_1\bar{c}_2=1$ and $\frac{1}{\bar{c}_2}=\frac{1+\bar{c}_1}{1+\bar{c}_2}$ are equivalent conditions. This splits our analysis into three branches: $\bar{c}_1\bar{c}_2 > 1$, $\bar{c}_1\bar{c}_2 < 1$, and $\bar{c}_1\bar{c}_2 = 1$.
\subsubsection{\textbf{Case 1:} $\bar{c}_1 \bar{c}_2 > 1$}
The optimal choice in the region $\frac{1}{\bar{c}_2}<\alpha<\frac{1+\bar{c}_1}{1+\bar{c}_2}$ is trivially $x_{\text{threshold}}$ since we know from Lemma~\ref{lma:x_hat} that $\bar{u}_2(\hat{x})=0$ and from Lemma~\ref{lma:x_thresh} that $\bar{u}_2(x_\text{threshold})$ is positive. From Lemma~\ref{lma:root_calculation}, we know that there will be no points where the optimal strategy changes in the range $\frac{1}{\frac{2}{1+\bar{c}_1}+\bar{c}_2}<\alpha\le\frac{1}{\bar{c}_2}$, and we know that $x_\text{threshold}$ will be the optimal strategy for $\alpha<\frac{1}{\frac{2}{1+\bar{c}_1}+\bar{c}_2}$ from Lemma~\ref{lma:x_hat}. Therefore, for this case, the optimal choice for any $\alpha<\frac{1+\bar{c}_1}{1+\bar{c}_2}$ is $x_{\text{threshold}}$.
\subsubsection{\textbf{Case 2: } $\bar{c}_1\bar{c}_2<1$}
By a similar argument as the first case, we reason that $x_2^*=\hat{x}$ will be the best choice for the region $\frac{1+\bar{c}_1}{1+\bar{c}_2}<\alpha<\frac{1}{\bar{c}_2}$. From Lemma~\ref{lma:root_calculation}, we know that the optimal strategy will switch only at $\alpha=\alpha^*$ in the region $\frac{1}{\frac{2}{1+\bar{c}_1}+\bar{c}_2}<\alpha<\frac{1}{\bar{c}_2}$ (we know from Lemma~\ref{lma:root_calculation} that $\alpha^*\ne\frac{1}{\bar{c}_2}$ because $\bar{c}_1\bar{c}_2\ne1$), and thus the optimal strategy will be $x_\text{threshold}$ for $\frac{1}{\frac{2}{1+\bar{c}_1}+\bar{c}_2}<\alpha<\alpha^*$. As with the previous case, we know from Lemma~\ref{lma:x_hat} that $x_\text{threshold}$ will also be the optimal solution for $\alpha\le\frac{1}{\frac{2}{1+\bar{c}_1}+\bar{c}_2}$ which completes all possible values for $\alpha$.
\subsubsection{\textbf{Case 3: } $\bar{c}_1\bar{c}_2=1$}
In this case, we have $\frac{1}{\bar{c}_2}=\frac{1+\bar{c}_1}{1+\bar{c}_2}$. From Lemma~\ref{lma:root_calculation}, $\alpha^*=\frac{1}{\bar{c}_2}=\frac{1+\bar{c}_1}{1+\bar{c}_2}$, so there are no turning points in the range $\frac{1}{\frac{2}{1+\bar{c}_1}+\bar{c}_2}<\alpha<\frac{1}{\bar{c}_2}$. From Appendix~\ref{subsec:plug_xth_xopt_cond_into_utility}, we know that $x_\text{threshold}$ is optimal at $\alpha=\frac{1}{\frac{2}{1+\bar{c}_1}+\bar{c}_2}$, so $x_\text{threshold}$ is the optimal solution for $\alpha<\frac{1}{\bar{c}_2}$, which is equivalent to our optimal strategies for the $\bar{c}_1 \bar{c}_2 > 1$ case.
It is worth noting that this is also equivalent to the optimal strategies in the $\bar{c}_1\bar{c}_2<1$ case, since no value for $\alpha$ satisfies $\alpha^*\le\alpha<\frac{1}{\bar{c}_2}$.
\end{proof}

\section{Optimal Utilities for Each Player}\label{sec:utility_discussion}
Now that we have the optimal strategies for each player from Theorems~\ref{thrm:player_1_strategy} and~\ref{thrm:player_2_strategy}, we will characterize each player's utility when they both play their optimal strategies.  We begin by re-defining Player 1's expected utility \textit{before} Player 2 has completed their turn. This means we must account for uncertainty in Player 2's success of their investment when evaluating Player 1's utility. As in Section~\ref{subsec:player_1_analysis}, we assume that Player 2 will attempt to invest some utility $x_2$. Thus, Player 1's expected utility for playing $x_1$, as evaluated prior to observing the outcome (success or failure of Player 2's investment) is given by
\begin{multline*}
E[u_1(x_1, x_2)]\\=
\begin{cases}
W_1(1-f_1(x_1))f_2(x_2)\\
\qquad + W_2(1-f_1(x_1))(1-f_2(x_2))\\
\qquad + (W_2-c_1)f_1(x_1),&0\le x_1<x_2, \\
W_1f_2(x_2)\\
\qquad +W_1(1-f_1(x_1))(1-f_2(x_2))\\
\qquad +(W_2-c_1)f_1(x_1)(1-f_2(x_2)),&x_1\ge x_2.
\end{cases}
\end{multline*}
Specifically, if Player 2 invests $x_2$, then for any investment $0 \le x_1 < x_2$ by Player 1, they will receive a prize of $W_1$ if they do not fail with probability $1-f_1(x_1)$ but Player 2 does. Otherwise, they will receive a prize of $W_2$ if neither player fails or Player 1 fails and pays a cost of $c_1$. This yields the expected utility in the first case shown above. On the other hand, if Player 1 chooses to invest an effort $x_1 \geq x_2$, then they will receive $W_1$ if neither player fails or only Player 2 fails. Otherwise, Player 1 will receive $W_2$ and pay a failure cost of $c_1$ if they fail with probability $f_1(x_1)$ and Player 2 succeeds.

As before, we can normalize this utility by defining
\begin{equation*}
    \bar{u}_1(x_1, x_2) \triangleq \frac{E[u_1(x_1, x_2)]-W_2}{W_1-W_2},
\end{equation*}
which yields:
\begin{multline*}
\bar{u}_1(x_1, x_2)\\ =
\begin{cases}
(1-f_1(x_1))f_2(x_2)\\
\qquad -\bar{c}_1f_1(x_1), &0\le x_1<x_2, \\ 
f_2(x_2)+(1-f_2(x_2))\\
\qquad[1-f_1(x_1)-\bar{c}_1f_1(x_1)], &x\ge x_2.
\end{cases}
\end{multline*}
This completes our definition for Player 1's utility before the start of the game.

The normalized utility for Player 2 before the start of the competition is identical to~\eqref{eqn:player_2_utility}, since they have no different information before the competition than when they play.
We will see that when the optimal strategies are used for each player, the resulting normalized utilities are only dependent on $\alpha$, $\bar{c}_1$, and $\bar{c}_2$, which we provide in the following corollary.

\begin{corollary}\label{corollary:cor1}
    When the optimal strategies for each player are played, the values of $\bar{u}_1(\alpha)=\bar{u}_1(x_1^*,x_2^*)$ and $\bar{u}_2(\alpha)=\bar{u}_2(x_2^*)$ are given in the following two cases. \\
    If $\bar{c}_1\bar{c}_2 \ge 1$:
\begin{align*}
\bar{u}_1(\alpha) &=
\begin{cases}
    \frac{\alpha}{1+\bar{c}_1}, & 0 \le \alpha < \frac{1+\bar{c}_1}{1+\bar{c}_2}, \\
    1, & \frac{1+\bar{c}_1}{1+\bar{c}_2} \le \alpha,
\end{cases} \\[1em]
\bar{u}_2(\alpha) &=
\begin{cases}
    1-\alpha\left(\frac{1+\bar{c}_2}{1+\bar{c}_1}\right), & 0 \le \alpha < \frac{1+\bar{c}_1}{1+\bar{c}_2}, \\
    0, & \frac{1+\bar{c}_1}{1+\bar{c}_2} \le \alpha.
\end{cases}
\end{align*}
    \noindent Otherwise, if  $\bar{c}_1\bar{c}_2 < 1$:
        \begin{align*}
        \bar{u}_1(\alpha) &=
        \begin{cases}
        \frac{\alpha}{1+\bar{c}_1}, & 0 \le \alpha < \alpha^*, \\
        1+\frac{1}{4\alpha}\left(\bar{c}_2^2\alpha^2-1\right)(1+\bar{c}_1), & \alpha^* \le \alpha < \frac{1}{\bar{c}_2}, \\
        1, & \frac{1}{\bar{c}_2} \le \alpha,
        \end{cases}\\
        \bar{u}_2(\alpha) &=
        \begin{cases}
        1-\alpha\left(\frac{1+\bar{c}_2}{1+\bar{c}_1}\right), & 0 \le \alpha < \alpha^*, \\
        \frac{1}{4\alpha}(1 - \alpha \bar{c}_2)^2, & \alpha^* \le \alpha < \frac{1}{\bar{c}_2}, \\
        0, & \frac{1}{\bar{c}_2} \le \alpha.
        \end{cases}
        \end{align*}
\end{corollary}
Before providing the proof of the corollary, we discuss the key insights the result provides.

\textbf{Key Insights:} 
We see that the utilities for each player are uniform over each of the differing strategy regions from Theorem~\ref{thrm:player_2_strategy}. 
We also see that each player's utility increases as their abilities increase relative to the other. Additionally, in all cases, the players' utilities are decreasing in their own failure cost. We note that in the case where $\bar{c}_1\bar{c}_2\ge1$, if Player 2 has a nonzero failure cost, Player 1 has a jump in utility as $\alpha$ increases past $\frac{1+\bar{c}_1}{1+\bar{c}_2}$, which is attributed to Player 2 folding. When Player 1 folds, Player 2's utility is increasing with $\bar{c}_1$. As $\bar{c}_2\to\infty$, $\bar{u}_1(\alpha)\to1$, and as $c_1\to\infty$, $\bar{u}_2(\alpha)\to1$ with finite values for other parameters, and thus each player's utility tends to increase with their opponent's failure cost.
We will now provide the proof of the corollary.

\begin{proof}
Recall that the strategies of Player 1 and Player 2 are entirely determined by the cases outlined in Theorem~\ref{thrm:player_2_strategy}, which can be split into three possible strategy configurations: either Player 1 will fold, Player 2 will fold, or neither Player will fold.

\subsection{Player 1 Folds}
From Theorem~\ref{thrm:player_2_strategy}, Player 2 will force Player 1 to fold in the $\bar{c}_1\bar{c}_2\ge1$ case if $0\le\alpha<\frac{1+\bar{c}_1}{1+\bar{c}_2}$, and in the $\bar{c}_1\bar{c}_2<1$ case if $0\le\alpha<\alpha^*$. Player 2's optimal effort to invest in this case is
\begin{equation*}
x_2^*=x_{1_{\text{max}}}\sqrt[r]{\frac{1}{1+\bar{c}_1}},
\end{equation*}
and Player 1's optimal effort from Theorem~\ref{thrm:player_1_strategy} is $x_1^*=0$. With these strategies,
\begin{align*}
    \bar{u}_1(\alpha) &= f_2\left(x_{1_{\text{max}}}\sqrt[r]{\frac{1}{1+\bar{c}_1}}\right) \\
    &= \frac{\alpha}{1+\bar{c}_1}, \\
    \bar{u}_2(\alpha) &= 1-(1+\bar{c}_2)f_2\left(x_{1_{\text{max}}}\sqrt[r]{\frac{1}{1+\bar{c}_1}}\right) \\
    &= 1-\alpha\left(\frac{1+\bar{c}_2}{1+\bar{c}_1}\right).
\end{align*}
\subsection{Player 2 Folds}
From Theorem~\ref{thrm:player_2_strategy}, Player 2 will fold in the $\bar{c}_1\bar{c}_2\ge1$ case if $\alpha\ge\frac{1}{\bar{c}_2}$, and in the $\bar{c}_1\bar{c}_2<1$ case if $\alpha\ge\frac{1+\bar{c}_1}{1+\bar{c}_2}$. Both players' optimal efforts to invest in this case are $x_1^*=x_2^*=0$. With these strategies:
\begin{align*}
    \bar{u}_1(\alpha) &= f_2(0) + (1-f_2(0))(1-f_1(0)-\bar{c}_1f_1(0)) = 1, \\
    \bar{u}_2(\alpha) &= f_1(0)(1-f_2(0))-\bar{c}_2f_2(0)= 0.
\end{align*}
\subsection{Neither Player Folds}
From Theorem~\ref{thrm:player_2_strategy}, only in the $\bar{c}_1\bar{c}_2<1$ case will neither player fold, if $\alpha^*\le\alpha<\frac{1}{\bar{c}_2}$. In this case both players' optimal efforts to invest are
\begin{equation*}
    x_1^*=x_2^*=x_{2_{\text{max}}}\sqrt[r]{\frac{1}{2}(1 - \bar{c}_2 \alpha)}.
\end{equation*}
With these strategies, we get
\begin{align*}
\bar{u}_1(\alpha) &= 1+f_1\left(x_{2_{\text{max}}}\sqrt[r]{\frac{1}{2}(1 - \bar{c}_2 \alpha)}\right) \times \\
&\qquad \left[\left(f_2\left(x_{2_{\text{max}}}\sqrt[r]{\frac{1}{2}(1 - \bar{c}_2 \alpha)}\right)-1\right)(1+\bar{c}_1)\right]\\
&=1+\frac{1}{4\alpha}\left(\bar{c}_2^2\alpha^2-1\right)(1+\bar{c}_1),
\end{align*}
and 
\begin{align*}
    \bar{u}_2(\alpha)&= f_1\left(x_{2_{\text{max}}}\sqrt[r]{\frac{1}{2}(1-\bar{c}_2 \alpha)}\right)\times \\
    &\qquad\left[1-f_2\left(x_{2_{\text{max}}}\sqrt[r]{\frac{1}{2}(1 - \bar{c}_2 \alpha)}\right)\right]\\
    &\qquad-c_2f_2\left(x_{2_{\text{max}}}\sqrt[r]{\frac{1}{2}(1 - \bar{c}_2 \alpha)}\right) \\
    &=\frac{1}{4\alpha}(1 - \alpha \bar{c}_2)^2.
\end{align*}
\end{proof}

\begin{table*}
\centering
\small
\begin{tabular}{c|c|c c|c c|c|c c|c|c c}
\hline
\text{Row} & \textbf{Case} & $x_{1_{\text{max}}}$ & $x_{2_{\text{max}}}$ & $\bar{c}_1$ & $\bar{c}_2$ & r &  $x_1^*$ & $x_2^*$ & \textbf{Outcome} & $\bar{u}_1(\alpha)$ & $\bar{u}_2(\alpha)$ \\
\hline
1
&$\bar{c}_1\bar{c}_2 \ge 1$
& 10 & 10 
& 10 & 10
& 10
& 7.868 & 0 
& Player 2 Folds 
& 1 & 0 \\

\hline
2
&$\bar{c}_1\bar{c}_2 \ge 1$
& 10 & 10 
& 5 & 10
& 10
& 8.360 & 0 
& Player 2 Folds 
& 1 & 0 \\

\hline
3
&$\bar{c}_1\bar{c}_2 \ge 1$
& 10 & 10 
& 10 & 5
& 10
& 0 & 7.868 
& Player 1 Folds 
& 0.091 & 0.456 \\

\hline
4
&$\bar{c}_1\bar{c}_2 \ge 1$
& 10 & 5 
& 10 & 10
& 10
& 7.868 & 0 
& Player 2 Folds 
& 1 & 0 \\

\hline
5
&$\bar{c}_1\bar{c}_2 \ge 1$
& 5 & 10 
& 10 & 10
& 10
& 0 & 3.934 
& Player 1 Folds 
& 0 & 0.999 \\

\hline

6
&$\bar{c}_1\bar{c}_2 \ge 1$
& 10 & 5 
& 100 & 5
& 10
& 6.303 & 0 
& Player 2 Folds 
& 1 & 0 \\

\hline

7
&$\bar{c}_1\bar{c}_2 \ge 1$
& 5 & 10 
& 5 & 100
& 10
& 0 & 4.180 
& Player 1 Folds 
& 0 & 0.984 \\
\hline

8
&$\bar{c}_1\bar{c}_2 \ge 1$
& 5 & 10 
& 5 & 100
& 1
& 0.833 & 0 
& Player 2 Folds 
& 1 & 0 \\

\hline
9
&$\bar{c}_1\bar{c}_2 < 1$
& 10 & 10 
& 0 & 0
& 10
& 9.330 & 9.330 
& Both Compete 
& 0.75 & 0.25 \\

\hline
10
&$\bar{c}_1\bar{c}_2 < 1$
& 10 & 10 
& 0 & 10
& 10
& 10 & 0 
& Player 2 Folds 
& 1 & 0 \\

\hline
11
&$\bar{c}_1\bar{c}_2 < 1$
& 10 & 10 
& 10 & 0
& 10
& 0 & 7.868 
& Player 1 Folds 
& 0.091 & 0.909 \\

\hline
12
&$\bar{c}_1\bar{c}_2 < 1$
& 10 & 5 
& 0 & 0
& 10
& 4.665 & 4.665 
& Both Compete 
& 0.999 & 0 \\

\hline
13
&$\bar{c}_1\bar{c}_2 < 1$
& 5 & 10 
& 0 & 0
& 10
& 0 & 5 
& Player 1 Folds 
& 0 & 0.999 \\

\hline
\end{tabular}

\caption{Numerical examples for various outcomes in the Stackelberg model.}\label{Table:Utility_Discussion}

\end{table*}

\subsection{Numerical Examples}
To gain further insight into the utilities and impact of failure probability and cost, we will compare numerical results from different scenarios provided from Table~\ref{Table:Utility_Discussion}. These results were calculated using Theorem~\ref{thrm:player_1_strategy}, Theorem~\ref{thrm:player_2_strategy}, and Corollary~\ref{corollary:cor1}. The first takeaway from rows 1 and 9 is that if Player 1 and Player 2 are equal players (they have the same $x_{i_{\text{max}}}$ and $\bar{c}_i$), then Player 1 tends to have higher utility. This indicates that when considering the probability of failure with an associated cost, it is better to be Player 1 in a competition of equal players. The second main takeaway is in the case of heterogeneous players, (the $x_{i_{\text{max}}}$ is different while $\bar{c}_i$ is the same or vice versa) represented in rows 2-5 and rows 10 and 11, the player with a lower cost or higher maximum effort tends to have a higher utility. However, in rows 6 and 7 we show that a significantly high failure cost can lower the optimal strategy value and corresponding utilities. Rows 7 and 8 demonstrate that increasing $r$ can change the outcome of a game from Player 1 folding to Player 2 folding, even when their relative effort and failure costs are the same. 
In rows 12 and 13, we show an example of the outcome when each player's failure costs are zero. The insights from Table~\ref{Table:Utility_Discussion} reflect those provided after each of the key results in our paper.

\section{Conclusion}\label{sec:Conclusion}
In this paper, we model a Stackelberg game in which two players (Player 1 and Player 2) invest effort while accounting for the probability of failure. The optimal strategies for the heterogeneous players are formulated in Theorems 1 and 2. The optimal strategies are a function of the relative abilities of the players, the individual costs of failure, and the shape of the failure function. Using these optimal strategies, we re-formulated the utilities accordingly for each player in terms of $\alpha$. We then provided a numerical discussion of the utilities and discussed the key insights to these results. In future work, it would be of interest to study the impact of information asymmetry in the players' maximum effort and cost as well as extending this game to more than two players.

\section{Appendix}\label{Appendix}

\subsection{Comparison of Utility when $x_{\text{threshold}}=\hat{x}$}\label{subsec:plug_xth_xopt_cond_into_utility}
We would like to compare $\bar{u}_2(x_\text{threshold})$ and $\bar{u}_2(\hat{x}_2^*)$ at the $\alpha$ where $x_\text{threshold}=\hat{x}_2^*$.
This occurs when $\alpha=\frac{1}{\frac{2}{1+\bar{c}_1}+\bar{c}_2}$ from \eqref{eqn:xhat_smaller_than_xth} and yields the following:
\begin{align*}
&-\alpha^2\left(\frac{1+\bar{c}_2}{1+\bar{c}_1}\right) + \alpha\ge\frac{1}{4}(1-\alpha\bar{c}_2)^2 \nonumber\\
&\Leftrightarrow \frac{1+\bar{c}_1}{2+\bar{c}_2+\bar{c}_1\bar{c}_2}\left(1-\frac{1+\bar{c}_1}{2+\bar{c}_2+\bar{c}_1\bar{c}_2}\left(\frac{1+\bar{c}_2}{1+\bar{c}_1}\right)\right)\ge \nonumber\\
&\qquad \frac{1}{4}\left(1-\frac{\bar{c}_2+\bar{c}_1\bar{c}_2}{2+\bar{c}_2+\bar{c}_1\bar{c}_2}\right)^2 \nonumber\\
&\Leftrightarrow \frac{1+\bar{c}_1}{2+\bar{c}_2+\bar{c}_1\bar{c}_2}\left(\frac{2+\bar{c}_2+\bar{c}_1\bar{c}_2-1-\bar{c}_2}{2+\bar{c}_2+\bar{c}_1\bar{c}_2}\right)\ge \nonumber\\
&\qquad \frac{1}{4}\left(\frac{2+\bar{c}_2+\bar{c}_1\bar{c}_2-\bar{c}_2-\bar{c}_1\bar{c}_2}{2+\bar{c}_2+\bar{c}_1\bar{c}_2}\right)^2 \nonumber\\
&\Leftrightarrow \frac{1+\bar{c}_1}{2+\bar{c}_2+\bar{c}_1\bar{c}_2}\left(\frac{1+\bar{c}_1\bar{c}_2}{2+\bar{c}_2+\bar{c}_1\bar{c}_2}\right)\ge\\
&\qquad\frac{1}{4}\left(\frac{2}{2+\bar{c}_2+\bar{c}_1\bar{c}_2}\right)^2 \nonumber\\
&\Leftrightarrow (1+\bar{c}_1)\left(\frac{1+\bar{c}_1\bar{c}_2}{2+\bar{c}_2+\bar{c}_1\bar{c}_2}\right)\ge\frac{1}{2+\bar{c}_2+\bar{c}_1\bar{c}_2} \nonumber\\
&\Leftrightarrow (1+\bar{c}_1)(1+\bar{c}_1\bar{c}_2)\ge1. \nonumber
\end{align*}
The left hand side is at least 1, so this inequality holds true.


\bibliographystyle{IEEEtran}
\bibliography{refs}

\end{document}